\documentclass[a4paper,11pt]{amsart}
\usepackage{amssymb, amsmath, amsthm,stmaryrd}
\usepackage{mdwlist}

\def\res{\hbox{ {\vrule height .3cm}{\leaders\hrule\hskip.3cm}}\hskip5.0\mu}

\newcommand\beqn{\begin{equation}}
\newcommand\eeqn{\end{equation}}
\newcommand\beqny{\begin{eqnarray}}
\newcommand\eeqny{\end{eqnarray}}
\newcommand\beqnyn{\begin{eqnarray*}}
\newcommand\eeqnyn{\end{eqnarray*}}

\usepackage{amsmath,amsfonts,amssymb,amsthm}
\newtheorem{theorem}{Theorem}[section]
\newtheorem{lemma}[theorem]{Lemma}
\newtheorem{proposition}[theorem]{Proposition}
\newtheorem{corollary}[theorem]{Corollary}

\def\a{\alpha}           
          
\def\g{\gamma}           
         
\def\e{\epsilon}         
\def\z{\zeta}            
\def\th{\theta}

\def\r{\rho} 

\def\s{\sigma}

\begin{document}
\setlength\parskip{5pt}
\title[maximum principle and unique continuation for singular minimal hypersurfaces]{A sharp strong maximum principle and a sharp unique continuation theorem for singular minimal hypersurfaces}
\author{Neshan Wickramasekera} 
\thanks{DPMMS, University of Cambridge, Cambridge, CB3 0WB, United Kingdom.}

\begin{abstract}
We prove the two theorems of the title, settling two long standing questions in the local theory of singular minimal hypersurfaces.  The sharpness of either result is with respect to its hypothesis on the size of the allowable singular sets. The proofs  of both theorems rely heavily on the author's recent regularity and compactness theory for stable minimal hypersurfaces, and on earlier work of Ilmanen, Simon and Solomon--White.
\end{abstract}

\maketitle
\newtheorem{hypotheses}{Hypotheses}
\renewcommand{\thehypotheses}{\arabic{section}.\arabic{hypotheses}}

\section{introduction and main results}\label{intro}
Let $N$ be an $(n+1)$-dimensional  smooth Riemannian manifold (without boundary, and not assumed to be complete). Consider two possibly singular minimal (i.e.\ area-stationary) hypersurfaces of $N$ with connected regular parts, a common point $x_{0}$ and with the property that locally near each of their common points, one hypersurface lies on one side of the other (the precise meaning  of which is that Hypothesis $K$ below is satisfied). If both hypersurfaces are free of singularities, it is a direct consequence of the Hopf maximum principle that  they must coincide. More generally, if we only assume that one hypersurface is free of singularities and if $x_{0}$ is a regular point of both, it is again straightforward to see that the hypersurfaces must coincide.\footnote{If we label the  hypersurfaces $M_{1}$, $M_{2}$ with $M_{2}$ free of singularities, then the Hopf maximum principle implies that 
the regular part of $M_{1}$, and hence $M_{1}$, is contained in $M_{2}.$ It is then an easy consequence of the first variation formula (see Lemma~\ref{UsefulLemma}) that $M_{1} = M_{2}.$}\\

Given the ubiquity of singular minimal hypersurfaces, it is a natural question to ask whether the same conclusion must hold if  the common point $x_{0}$ is  a \emph{singular} point. This question is  much more subtle and has been studied, in various special cases, by a number of authors. In view of a theorem of Solomon and White (\cite{SolWhi}, with an improvement due to White (\cite{W}, Theorem~4))---stated as Theorem~\ref{SolWhi} below---we know that in case one of the two hypersurface is free of singularities, we can always conclude (without assuming that $x_{0}$ is a regular point of both) that the hypersurfaces must coincide.  Earlier work of Miranda (\cite{M}) had established this for two oriented area minimizing boundaries one of which is free of singularities. Thus if the common point $x_{0}$ \emph{is} a singular point of one of the hypersurfaces, then it must necessarily be a singular point of \emph{both}.  Moschen in \cite{MM} and Simon in \cite{S}  independently proved that the answer to the above question is yes in case the hypersurfaces are oriented and area minimizing (with no restriction on either of their singular sets beyond 
what is imposed by the area minimizing property which implies that the singular sets must have Hausdorff dimension $\leq n-7$). Ilmanen in \cite{I} generalised this result to stationary hypersurfaces both of which are allowed to be singular but with the restriction (rather strong for stationary hypersurfaces) that their singular sets have locally finite $(n-2)$-dimensional Hausdorff measure. For other strong maximum principle type results in the presence of singularities,  see \cite{Sch} where certain singular hypersurfaces with non-zero mean curvature are considered, and \cite{W} where certain varieties of arbitrary codimension are considered.

Here we show  that the theorem of Solomon and White and that of Ilmanen (and hence all of the above results for minimal hypersurfaces) can be sharpened into a single strong maximum principle (Theorem~\ref{maxm} below) which says that the only a priori regularity hypothesis needed to conclude that two $n$-dimensional stationary hypersurfaces as above must coincide  is that the singular set of \emph{one} of them has $(n-1)$-dimensional Hausdorff measure zero. This condition is sharp in the sense that a larger singular set cannot be allowed (see Remark (1) below).

The precise notion of ``possibly singular  minimal submanifolds''  we use in Theorem~\ref{maxm} is that of stationary integral varifolds. Since our theorem  requires  only stationarity of the hypersurfaces and very little a priori regularity, varifolds are the natural (and most general) context for it. See Section~\ref{not} below for the basic definitions concerning varifolds needed in this paper and explanation of notation we use. We refer the reader to \cite{AW} or [\cite{S1}, Chapter 8] for a  detailed account of the theory of stationary varifolds. 

Before we can state our strong maximum principle, we need to make precise the notion, given two codimension 1 varifolds, that one of them lies locally on one side of the other near a common point. If both varifolds are free of singularities, the meaning of this is clear since a regular hypersurface  divides into exactly two components any sufficiently small geodesic ball of the ambient manifold centered at a point on the hypersurface. In the presence of singularities, it is natural to adopt a similar criterion but  insist that it holds only for every common point which is a regular point of at least one varifold. Thus we introduce the following terminology:

Let $V_{1}$, $V_{2}$ be codimension 1 varifolds on $N.$ We say that \emph{${\rm spt} \, \|V_{2}\|$ lies locally on one side of ${\rm reg} \, V_{1}$} if the following holds:

\noindent
{\bf Hypothesis K:} For every point $y \in {\rm reg} \, V_{1} \cap {\rm spt} \, \|V_{2}\|$, there exists $\r>0$ such that (i) ${\rm sing} \, V_{1} \cap B_{\r}(y) = \emptyset,$ (ii) $B_{\r}(y) \setminus {\rm spt} \, \|V_{1}\|$ is disconnected and (iii) ${\rm spt} \, \|V_{2}\| \cap B_{\r}(y)$ is contained in the closure of one of the connected components of $B_{\r}(y) \setminus {\rm spt} \, \|V_{1}\|$.

Since $y \in {\rm reg} \, V_{1}$, the requirements (i) and (ii) in Hypothesis K are of course automatically satisfied if $\r = \r(y)>0$ is sufficiently small. (Our choice of terminology is based on the fact that letter $K$ consists of a regular piece and a singular piece with the singular  piece on one side of the regular piece!)

\smallskip

\begin{theorem}[Strong maximum principle for singular minimal hypersurfaces]\label{maxm}
Let $V_{1}$, $V_{2}$ be stationary codimension 1 integral $n$-varifolds on a smooth Riemannian manifold such  that ${\rm spt} \, \|V_{j}\|$ is connected for $j=1, 2.$ If 
\begin{itemize}
\item[(i)] ${\rm spt} \, \|V_{2}\|$ lies locally on one side of ${\rm reg} \, V_{1}$ (in the sense that Hypothesis K above holds) and
\item[(ii)]  ${\mathcal H}^{n-1} \, ({\rm sing} \, V_{1}) = 0$, 
\end{itemize}
then either ${\rm spt} \, \|V_{1}\| \cap {\rm spt} \, \|V_{2}\| = \emptyset$ or  ${\rm spt} \, \|V_{1}\| = {\rm spt} \, \|V_{2}\|.$ 
\end{theorem}

\noindent
{\bf Remarks:} {\bf (1)} The theorem is sharp with respect to the singular set hypothesis ${\mathcal H}^{n-1} \, ({\rm sing} \, V_{1}) = 0$ in the sense that it cannot be weakened to  ${\mathcal H}^{n-1+\e} \, ({\rm sing} \, V_{1}) = 0$ for any $\e > 0$;
to see this,  consider for instance four planes in ${\mathbb R}^{3}$ with a common axis and with at least three of the planes distinct, and let $V_{1}$ be the union of an ``inner''  pair of hyperplanes and $V_{2}$ be the union of the corresponding ``outer'' pair (each with multiplicity 1).  A similar counterexample in which neither varifold is the sum of two non-trivial stationary varifolds and their regular parts have non-empty intersection is obtained by taking $V_{1},$ $V_{2}$ to be the multiplicity 1 varifolds supported on three and, respectively  five, equally spaced half-planes meeting along a common axis.

\noindent
{\bf (2)}  The theorem is also sharp with respect to hypothesis (i) in the sense that it is not enough to require merely that ${\rm reg} \, V_{2}$ lies locally on one side of ${\rm reg} \, V_{1}$ (that is, to require in Hypothesis K merely that $y \in {\rm reg} \, V_{1} \cap {\rm reg} \, V_{2}$); to see this,  let $V_{2}$  be the multiplicity 1 varifold supported on the union of three equally spaced half-planes in ${\mathbb R}^{3}$ meeting along a common axis, and $V_{1}$ be the multiplicity 1 varifold supported on  the plane containing one of the three half-planes of $V_{2}$. 

\noindent
{\bf (3)} If $\Omega_{1}$, $\Omega_{2}$ are open subsets of $N$ with $\Omega_{1} \subset \Omega_{2}$ and $V_{j} = |\partial \, \llbracket\Omega_{j}\rrbracket|$ for $j=1, 2$ (i.e.\ $V_{j}$ is the 
$n$-varifold corresponding to the multiplicity 1 boundary of the $(n+1)$-dimensional current defined by  the open set $\Omega_{j}$), then ${\rm spt} \, \|V_{2}\|$ lies locally on one side of ${\rm reg} \, V_{1},$ and also
${\rm spt} \, \|V_{1}\|$ lies locally on one side of ${\rm reg} \, V_{2}.$ So in this case, if $V_{1}$, $V_{2}$ are stationary; ${\rm spt} \, \|V_{j}\|$ is connected for $j=1, 2$;  ${\rm spt} \, \|V_{1}\| \cap {\rm spt} \, \|V_{2}\| \neq \emptyset$ and \emph{either ${\mathcal H}^{n-1} \, ({\rm sing} \, V_{1}) = 0$ or 
${\mathcal H}^{n-1} \, ({\rm sing} \, V_{2}) = 0$}, then it follows (from the theorem) that ${\rm spt} \, \|V_{1}\| = {\rm spt} \, \|V_{2}\|.$ 

\noindent
{\bf (4)} Allowed in Hypothesis K is the possibility that ${\rm reg} \, V_{1} \cap {\rm spt} \, \|V_{2}\| = \emptyset$, in which case, subject also to all other hypotheses of the theorem, the conclusion is that ${\rm spt} \, \|V_{1}\| \cap {\rm spt} \, \|V_{2}\| = \emptyset.$

We also prove a sharp unique continuation result (Theorem~\ref{uc} below) for stationary codimension 1 integral varifolds, the key to which are also the tools establishing Theorem~\ref{maxm}. Recall that the classical weak unique continuation property for solutions to the minimal surface system implies that if $M_{1}$, $M_{2}$ are $k$-dimensional connected, smoothly embedded minimal submanifolds of a Riemannian manifold $N$ with ${\rm sing} \, M_{j} \equiv {\rm clos} \, M_{j} \setminus M_{j} = \emptyset$ for $j=1, 2,$ and if $M_{1} \cap M_{2}$ has non-empty interior (as a subset of $M_{1}$ or $M_{2}$),  then $M_{1} = M_{2}.$ For stationary varifolds with connected supports, this statement generally is false even in codimension 1. However, we have the following:
\begin{theorem}[Unique continuation for singular minimal hypersurfaces]\label{uc}
Let $V_{1}$, $V_{2}$ be stationary codimension 1 integral $n$-varifolds on a smooth Riemannian manifold such that ${\rm spt} \, \|V_{j}\|$ is connected and ${\mathcal H}^{n-1}({\rm sing} \, V_{j}) = 0$ for $j=1, 2.$ If ${\mathcal H}^{n-1+\g} \, ({\rm spt} \, \|V_{1}\| \cap {\rm spt} \, \|V_{2}\|) > 0$ for some $\g \in (0, 1]$ then ${\rm spt} \, \|V_{1}\| = {\rm spt} \, \|V_{2}\|.$ 
\end{theorem}

\noindent
{\bf Remark:} Evidently, in this theorem, the singular set hypothesis ${\mathcal H}^{n-1} \, ({\rm sing} \, V_{1}) = {\mathcal H}^{n-1} \, ({\rm sing} \, V_{2}) =0$ is sharp in the sense that it cannot be weakened to ${\mathcal H}^{n-1+\e} \, ({\rm sing} \, V_{i}) = {\mathcal H}^{n-1} \, ({\rm sing} \, V_{j}) = 0$ for any $\e >0$ and $(i, j) = (1, 2)$ or $(i, j) = (2, 1)$; to see this, consider for instance the example described in  Remark (2) following the statement of Theorem~\ref{maxm}.  It is also clear that the theorem does not hold with $\g = 0.$

\section{notation}\label{not}
Throughout the paper, we use definitions of \cite{AW} (also of \cite{S1}, Chapter 8) with regard to varifolds. Thus, assuming without loss of generality that $N$ is properly embedded in ${\mathbb R}^{n+k}$ for some fixed $k \geq 1$, an $n$-varifold on $N$ is a Radon measure on $N \times G(n,k) \cap \{(x, S) \, : \, S\subset T_{x} \, N\},$ where $G(n,k)$ is the Grassmannian of unoriented $n$-dimensional subspaces of ${\mathbb R}^{n+k}$; 

For an $n$-varifold $V$ on $N$, 
$\|V\|$ (which, in the notation of \cite{S1}, is $\mu_{V}$) denotes the Radon measure induced on $N$ via $\|V\|(A) = V(A \times G(n,k) \cap \{(x,S) \, : \,  S \subset T_{x}\, N\});$ ${\rm reg} \, V$ is the regular part of $V,$ defined to be the set of points $x \in {\rm spt} \, \|V\|$ such that ${\rm spt} \, \|V\|$ is an $n$-dimensional properly embedded submanifold of $N$ near $x;$ ${\rm sing} \, V$ is the singular set of $V$, defined by ${\rm sing} \, V = {\rm spt} \, \|V\| \setminus {\rm reg} \, V;$ 

An $n$-varifold $V$ on $N$ is \emph{integral} if there is an ${\mathcal H}^{n}$ measurable,  countably $n$-rectifiable set $M \subset N$ and a locally ${\mathcal H}^{n}$ integrable function $\th$ (the multiplicity function) on $M$ with $\th(x)$ a positive integer for ${\mathcal H}^{n}$-a.e. $x \in M$ such that 
$$V(\varphi) \equiv \int_{N \times G(n,k) \cap \{(x, S) \, : \, S\subset T_{x} \, N\}} \varphi(x, S) \, dV(x, S) = \int_{M} \varphi(x, T_{x} \, M) \th(x)\, d{\mathcal H}^{n}(x)$$ 
for every $\varphi \in C_{c}(N \times G(n,k) \cap \{(x, S) \, : \, S\subset T_{x} \, N\});$ here $T_{x} \, M$ denotes the approximate tangent space to $M$ at $x$ (which exists for ${\mathcal H}^{n}$-a.e.\ $x \in M$), and ${\mathcal H}^{n}$ is the $n$-dimensional Hausdorff measure on $N$; 

The notion of integral $n$-varifold obviously generalises the notion of $n$-dimensional $C^{1}$ submanifold. For an $n$-dimensional  $C^{1}$ submanifold $M$ of $N$, we let $|M|$ denote the associated multiplicity 1 varifold, defined by 
$|M|(\varphi) = \int_{M} \varphi(x, T_{x} \, M) \, d{\mathcal H}^{n}(x)$ for each $\varphi \in C_{c}(N \times G(n,k) \cap \{(x, S) \, : \, S\subset T_{x} \, N\}).$ 

We shall also use the following notation throughout the paper: 

For $y \in N$ and $\r>0$, we let $B_{\r}(y)$ denote the open geodesic ball in $N$ with centre $y$ and radius $\r$; For $A \subset N$, ${\rm clos} \, A$ denotes the closure of $A$ in $N$; 

For $s >0$, ${\mathcal H}^{s}_{\infty}$ denotes the outer measure on $N$ defined by ${\mathcal H}^{s}_{\infty}(A) = \inf \, \{\sum_{j=1}^{\infty} \, \left(\frac{{\rm diam} \, \Omega_{j}}{2}\right)^{s} \, : \Omega_{1}, \Omega_{2}, \ldots  \subset N, \;\;  \, A \subset \cup_{j=1}^{\infty} \Omega_{j}\}$.

\section{Connection to the regularity theory for stable minimal hypersurfaces and to other previous work}\label{previous} 
Our proofs of Theorems~\ref{maxm} and  ~\ref{uc} (given in Sections~\ref{maxm-proof}, ~\ref{uc-proof} below) depend in an essential way on the recently established sharp regularity and compactness theory for stable minimal hypersurfaces (\cite{Wic}; see Theorem~\ref{reg} below). They also rely on the results and ideas contained in  the aforementioned work \cite{S}, \cite{SolWhi}, \cite{I}, which establish various special cases of Theorem~\ref{maxm}. Accordingly, in this section and the next, we collect these earlier results and briefly explain their role in the proofs of Theorems~\ref{maxm} and ~\ref{uc}.

First, in our proof of Theorem~\ref{maxm}, we shall use Theorem~\ref{SolWhi} below, which is the special case of Theorem~\ref{maxm} when $V_{1}$ is free of singularities.

\begin{theorem}[\cite{SolWhi}, with improvement as in  (\cite{W}, Theorem 4)]\label{SolWhi}
Let $V_{1}$, $V_{2}$ be stationary codimension 1 integral varifolds on an open geodesic ball $B$ of a Riemannian manifold $N$ such that ${\rm spt} \, \|V_{j}\|$ is connected for $j=1, 2$. If ${\rm sing} \, V_{1} = \emptyset,$ $B \setminus {\rm spt} \, \|V_{1}\|$ is disconnected and ${\rm spt} \, \|V_{2}\|$ is contained in the closure of a connected component of $B \setminus {\rm spt} \, \|V_{1}\|$,  then either ${\rm spt} \, \|V_{1}\| \cap {\rm spt} \, \|V_{2}\|\ \cap B  = \emptyset$ or ${\rm spt} \, \|V_{1}\| \cap B = {\rm spt} \, \|V_{2}\| \cap B$. 
\end{theorem}

The other key ingredient used in proofs of both Theorem~\ref{maxm} and Theorem~\ref{uc} is Theorem~\ref{disjoint} below, which is perhaps also of independent interest. It says that two stationary codimension 1 integral $n$-varifolds with connected supports intersecting on a set of $(n-1)$-dimensional  Hausdorff measure zero must in fact have disjoint supports.

\begin{theorem}\label{disjoint}
Let $V_{1}$, $V_{2}$ be stationary codimension 1 integral $n$-varifolds on a smooth Riemannian manifold. If ${\mathcal H}^{n-1}({\rm spt} \, \|V_{1}\| \cap {\rm spt} \, \|V_{2}\|) = 0$ then 
${\rm spt} \, \|V_{1}\|$ and ${\rm spt} \, \|V_{2}\|$ are disjoint.  
\end{theorem}

A pair of transversely intersecting hyperplanes in a Euclidean space shows that the hypothesis ${\mathcal H}^{n-1} \, ({\rm spt} \, \|V_{1}\| \cap {\rm spt} \, \|V_{2}\|) = 0$ in this theorem is sharp. 
 
Ilmanen in \cite{I} proved Theorem~\ref{disjoint} subject to the stronger hypothesis that ${\mathcal H}^{n-2} \, ({\rm spt} \, \|V_{1}\| \cap {\rm spt} \, \|V_{2}\| \cap K) < \infty$ 
for every compact set $K \subset N$, and used it to deduce the special case of Theorem~\ref{maxm} when both varifolds have singular sets of locally finite $(n-2)$-dimensional Hausdorff measure. Among the key ingredients of Ilmanen's proof is the regularity and compactness theory of Schoen and Simon (\cite{SS}, Theorem 1 and Theorem 2) for stable minimal hypersurfaces, which requires a priori knowledge that the singular sets of the ($n$-dimensional) hypersurfaces have locally finite $(n-2)$-dimensional Hausdorff measure.

In recent work \cite{Wic}, a sharp regularity and compactness theory for stable hypersurfaces has been established generalizing the Schoen--Simon theory, and this generalization  is key to Theorem~\ref{disjoint}. We shall discuses Theorem~\ref{disjoint} in more detail in the next section, and devote the remainder of this section to a brief discussion of results in \cite{Wic}.  

In order to explain the main content of the work \cite{Wic}, let us make the following two definitions: 

\noindent
{\bf Definition:} Given an $n$-varifold $V$ on a manifold $N$,  a point $y \in {\rm sing} \, V$ is said to be a \emph{classical singularity} of $V$ if there exists $\r>0$ and $\a \in (0, 1)$ such that 
${\rm spt} \, \|V\| \cap B_{\r}(y)$ is equal to the union of three or more $n$-dimensional embedded $C^{1, \a}$ hypersurfaces-with-boundary in $B_{\r}(y),$ with a common $C^{1, \a}$ boundary containing $y,$ and such that the hypersurfaces-with-boundary meet pairwise only along the common boundary. 

\noindent
{\bf Remark:} If $V$ is stationary, it follows from the Hopf boundary point lemma for divergence form operators (\cite{FG}; see also \cite{HS}) that any two distinct, adjacent hypersurfaces-with-boundary 
meeting along their common boundary as in the definition of classical singularity must do so transversely. Equivalently, the number of distinct half-hyperplanes of the (unique) tangent cone at a classical singularity $y$ is the same as the number of distinct hypersurfaces-with-boundary corresponding to $y$. 

\noindent
{\bf Definition:} A stationary integral $n$-varifold $V$ on an $(n+1)$-dimensional Riemannian manifold $N$ is said to be \emph{stable} if  for each sufficiently small geodesic ball $B\subset N$ and any open ball $\widetilde{B} \subset B$ with ${\rm dim}_{\mathcal H} \, ({\rm sing} \, V \cap \widetilde{B}) \leq n-7$ in case $n \geq 7$ or ${\rm sing} \, V \cap \widetilde{B} = \emptyset$ in case $n \leq 6$, the stability inequality 
$$\hspace{1in} \int_{{\rm reg} \, V \cap \widetilde{B}} (|A|^{2} + {\rm Ric}_{N} \, (\nu))\z^{2} \leq \int_{{\rm reg} \, V \cap \widetilde{B}} |\nabla \, \z|^{2}\hspace{1.2in} (\star)$$ 
holds for every $\zeta \in C^{1}_{c}({\rm reg} \, V \cap \widetilde{B}).$ Here $A$ is the second fundamental form of ${\rm reg} \, V$, $\nu$ is a continuous choice of unit normal to ${\rm reg} \, V \cap \widetilde{B}$ and ${\rm Ric}_{N}(\nu)$ denotes the Ricci curvature of $N$ in the direction of $\nu.$ Stability of $V$ is equivalent to requiring that 
for each open ball $\widetilde{B} \subset B$ as above, 
$V$ has non-negative second variation with respect to area for deformations by ambient vector fields 
with compact support $\subset \widetilde{B} \setminus {\rm sing} \, V$ and normal to ${\rm reg} \, V \cap \widetilde{B}$ at points of ${\rm reg} \, V \cap \widetilde{B}$. 

The work in \cite{Wic} shows that the same regularity and compactness conclusions as in \cite{SS} can be made for stable codimension 1 integral varifolds without any hypothesis on the singular set beyond the (obviously necessary) requirement  that there are no classical singularities. More precisely, we have the following:

\begin{theorem}[\cite{Wic}, Theorem 18.1]\label{reg}
 If a stable codimension 1 integral $n$-varifold on a Riemannian manifold has no classical singularities, then its singular set 
is empty if $n \leq 6$, discrete if $n=7$ and has Hausdorff dimension at most $n-7$ if $n \geq 8;$
moreover, each uniformly mass bounded subset of the class of stable codimension 1 integral varifolds with no classical singularities is compact in the topology of varifold convergence.  
\end{theorem}

There is also a ``Sheeting Theorem,'' namely [\cite{Wic}, Theorem 18.2], for the class of stable codimension 1 integral varifolds having no classical singularities, which implies that a sequence of stable codimension 1 integral varifolds with no classical singularities  converging weakly to a smooth limit must converge smoothly with multiplicity $\geq 1$. 

\noindent
{\bf Remarks:} {\bf (1)} In particular, the regularity and compactness conclusions of Theorem~\ref{reg} and the regularity conclusions of the Sheeting Theorem [\cite{Wic}, Theorem 18.2] hold for stable codimension one $n$-varifolds having singular sets of $(n-1)$-dimensional Hausdorff measure zero, since such varifolds must automatically satisfy the no-classical-singularities hypothesis. (It is this special case of Theorem~\ref{reg} that is needed for the purposes of the present paper; the proofs of Theorem~\ref{reg} or the Sheeting Theorem [\cite{Wic}, Theorem~18.2] 
in this special case however are only marginally simpler than the general case). 

\noindent
{\bf (2)} By the Remark following the definition of classical singularity, for a stationary varifold, non-existence of tangent cones supported on unions of three or more half-hyperplanes meeting along a common axis implies non-existence of classical singularities. For stable codimension 1 integral varifolds, non-existence of such tangent cones is in fact equivalent to non-existence of classical singularities. (These facts however are not needed in the present paper.) What one might call a quantitative version of  this  statement is established in \cite{Wic} as the \emph{Minimum Distance Theorem} [\cite{Wic}, Theorem 3.4], and it plays a very important auxiliary role  in the proof of Theorem~\ref{reg}. The Minimum Distance Theorem says that given a stationary cone ${\mathbf C}$ supported on the union  of three or more distinct $n$-dimensional half-hyperplanes meeting along a common boundary, a stable codimension 1 integral $n$-varifold $V$ with no classical singularities cannot be too close to ${\mathbf C}$ at unit scale. Note that such a theorem would be an easy consequence of the Sheeting Theorem if the singular set of $V$ is sufficiently small to be a removable set for the stability inequality (precisely, as small as having locally finite $(n-2)$-dimensional Hausdorff measure). Theorem~\ref{reg} however makes no hypothesis on the size of the singular sets of stable varifolds, and therefore there is little hope of proving the Sheeting Theorem and the Minimum Distance Theorem  independently of each other or even sequentially one after the other. Instead, the strategy adopted in \cite{Wic} is to prove both the Sheeting Theorem and the Minimum Distance Theorem \emph{simultaneously} by an inductive argument.

\section{Proof of Theorem~\ref{disjoint}}\label{disj}
Ilmanen's argument in \cite{I}  taken with Theorem~\ref{reg} and the Sheeting Theorem [\cite{Wic}, Theorem 18.2]  respectively in place of [\cite{SS}, Theorem 2] and [\cite{SS}, Theorem 1]  yields Theorem~\ref{disjoint}. The argument consists of two steps, which we now describe briefly in the context of Theorem~\ref{disjoint}, referring the reader to \cite{I} for details:

\noindent
{\bf Step 1:} Prove the special case of the theorem when $V_{j},$ for $j=1, 2$, is stable
 with ${\rm spt} \, \|V_{j}\|$ connected and ${\mathcal H}^{n-1} \, ({\rm sing} \, V_{j})  = 0.$    
This is where Theorem~\ref{reg} and the Sheeting Theorem [\cite{Wic}, Theorem 18.2] are essential. Note that Theorem~\ref{reg} tells us in particular that if $V_{j}$ is as above, then ${\rm sing} \, V_{j} = \emptyset$ if $n \leq 6$ and ${\mathcal H}^{n-7+\e} \, ({\rm sing} \, V_{j}) = 0$ for each $\e>0$ if $n \geq 7.$ 

For this step, Ilmanen uses  a certain ``Jacobi field argument''  due to Simon (\cite{S}). This Jacobi field argument was the main idea in Simon's proof of Theorem~\ref{maxm} in case $V_{1}$, $V_{2}$ are oriented area minimizing hypersurfaces, and in its original form, the argument showed that whenever the two minimizing hypersurfaces have a common tangent cone $C$ at a common singular point $x_{0}$ near which their regular parts are disjoint, the positive Jacobi fields along $C$ produced by rescaling, about $x_{0}$, the difference of (signed) height of the hypersurfaces relative to $C$ would contradict  the Bombieri--Giusti Harnack inequality (\cite{BG}) for non-negative superharmonic functions on $C$. (To rule out altogether the possibility of intersection without coincidence, Simon then argued that reduction to the case of common tangent cones is always possible.) This  argument  crucially relied on the fact that the  hypersurfaces belong to a compact class of minimal varieties with \emph{sufficiently small} singular sets and for which there is a Sheeting Theorem (which says that whenever a hypersurface in the class is Hausdorff close to a smooth element in the class, it is $C^{2}$ close to the smooth element in the interior)---all of which are guaranteed by the well-known regularity and compactness theory for codimension 1 area minimizing rectifiable currents. 

In view of Theorem~\ref{reg} and the Sheeting Theorem [\cite{Wic}, Theorem 18.2] which provide the necessary regularity and compactness properties for stable hypersurfaces $V$ satisfying ${\mathcal H}^{n-1} \, ({\rm sing} \, V) = 0$, we can use Simon's Jacobi field argument exactly as it was used in \cite{I} (see \cite{I}, Lemmas 2-5)  to establish step 1. In place of  the Bombieri--Giusti Harnack inequality which is not known to extend to stable hypersurfaces of dimension $\geq 7$, a 
mean value inequality  (\cite{I}, Lemma 4) for super harmonic functions on stationary cones with sufficiently small singular sets was established and used in \cite{I}; this result is also applicable in the present setting where 
the necessary lower dimensionality of the singular sets of the relevant cones (which arise as tangent cones to stable hypersurfaces)  is guaranteed by Theorem~\ref{reg}.

\noindent
{\bf Step 2:} The second step in the proof of Theorem~\ref{disjoint} is to interpose, following exactly the construction  in [\cite{I}, Proof of Theorem A], two \emph{stable} integral $n$-varifolds $W_{1}$, $W_{2}$  on $N$ (which in \cite{I} are labeled $N$, $N^{\prime}$) ``between'' the stationary ones $V_{1}$, $V_{2}$ as in Theorem~\ref{disjoint}. This is done as follows: 

First construct $W_{1}$ such that 
$W_{1} \res (N \setminus ({\rm spt} \, \|V_{1}\| \cap {\rm spt} \, \|V_{2}\|))$ is stationary in 
$N \setminus ({\rm spt} \, \|V_{1}\| \cap {\rm spt} \, \|V_{2}\|)$ and 
$${\rm spt} \, \|V_{1}\| \cap {\rm spt} \, \|V_{2}\| \subset {\rm spt} \, \|W_{1}\| \cap {\rm spt} \, \|V_{1}\| \subset {\rm sing} \, W_{1}.$$ 
$W_{1}$ is obtained as the weak limit of solutions to a sequence of certain obstacle problems, constructed using ${\rm spt} \, \|V_{1}\| \cup {\rm spt} \, \|V_{2}\|$ as a barrier (see [\cite{I}, Lemma 7] and [\cite{I}, Proof of Theorem A, step 1]), so it follows from \cite{SS}  that $W_{1} \res (N \setminus ({\rm spt} \, \|V_{1}\| \cap {\rm spt} \, \|V_{2}\|))$ is stable in $N \setminus {\rm spt} \, \|V_{1}\| \cap {\rm spt} \, \|V_{2}\|$ with 
$${\rm dim}_{\mathcal H} \, ({\rm sing} \, W_{1} \setminus {\rm spt} \, \|V_{1}\| \cap {\rm spt} \, \|V_{2}\|) \leq n-7.$$ (Note that we do not need Theorem~\ref{reg} here since we have, by the minimizing property satisfied by the solutions to the obstacle problem, the a priori regularity necessary to apply
 \cite{SS}.)

 Using the fact that  
${\mathcal H}^{n-1} \, ({\rm spt} \, \|V_{1}\| \cap {\rm spt} \, \|V_{2}\|) = 0$, it can be shown that $W_{1}$ is stationary in $N$ (\cite{I}, Proof of Theorem A, step 2). We also have by the dimension estimate above that ${\mathcal H}^{n-1} \, ({\rm sing} \, W_{1}) = 0,$ and hence that ${\mathcal H}^{n-1} \, ({\rm spt} \, \|W_{1}\| \cap {\rm spt} \, \|V_{1}\|) = 0.$

Thus, we can repeat the process with $W_{1}$ in place of $V_{1}$ and $V_{1}$ in place of $V_{2}$ to construct $W_{2},$ so that  $W_{2}$ is stationary in $N,$ stable in  $N \setminus ({\rm spt} \, \|W_{1}\| \cap {\rm spt} \, \|V_{1}\|)$ with 
$${\rm spt} \, \|V_{1}\| \cap {\rm spt} \, \|W_{1}\| \subset {\rm spt} \, \|W_{2}\| \cap {\rm spt} \, \|W_{1}\| \subset {\rm sing} \, W_{2}$$
and ${\mathcal H}^{n-1} \, ({\rm sing} \, W_{2})  = 0.$ 

It follows that ${\rm spt} \, \|V_{1}\| \cap {\rm spt} \, \|V_{2}\| \subset {\rm spt} \, \|W_{1}\| \cap {\rm spt} \, \|W_{2}\|$ and ${\mathcal H}^{n-1} \, ({\rm spt} \, \|W_{1}\| \cap {\rm spt} \, \|W_{2}\|) = 0,$ which contradicts  the assertion of  step 1 unless ${\rm spt} \, \|V_{1}\| \cap {\rm spt} \, \|V_{2}\| = \emptyset$.

\noindent
{\bf Remark:} We wish to point out a subtlety in the way the stability hypothesis needs to be verified  (in any situation, and in particular for $W_{1},$ $W_{2}$ as above) when applying Theorem~\ref{reg} or the Sheeting Theorem [\cite{Wic}, Theorem 18.2]. 

Given a stationary integral $n$-varifold $V$ on an $(n+1)$-dimensional Riemannian manifold $N$ such that $V$ has no classical singularities  (or satisfies the condition ${\mathcal H}^{n-1} \, ({\rm sing} \, V) = 0$, which, as mentioned above,  is the special case of the no-classical-singularities hypothesis relevant to this paper), Theorem~\ref{reg} and the Sheeting Theorem [\cite{Wic}, Theorem 18.2] require stability of \emph{every} region of ${\rm spt} \, \|V\|$ in which the singular set ${\rm sing} \, V$ has Hausdorff dimension $\leq n-7$ in case $n \ge 7$ or is empty in case $n \leq 6$; more precisely, the theorems require such $V$ to satisfy the hypothesis that for each sufficiently small geodesic ball $B = B_{\r}(y) \subset N$ and any open ball $\widetilde{B} \subset B$ with ${\rm dim}_{\mathcal H} \, ({\rm sing} \, V \cap \widetilde{B}) \leq n-7$ in case $n \geq 7$ or ${\rm sing} \, V \cap \widetilde{B} = \emptyset$ in case $n \leq 6$, the stability inequality $(\star)$ holds for every $\zeta \in C^{1}_{c}({\rm reg} \, V \cap \widetilde{B}).$

In particular, it is \emph{not} enough to verify stability away from a closed set $\Sigma$ having $(n-1)$-dimensional Hausdorff measure zero unless $\Sigma \subset {\rm sing} \, V$. 
That is to say, in case  there is a closed set 
$\Sigma \subset {\rm spt} \, \|V\|$ with ${\mathcal H}^{n-1} \, (\Sigma) = 0$ such that $V$ is stable away from $\Sigma$ (in the sense that $(\star)$ holds for every open ball $\widetilde{B} \subset B_{\r}(y) \setminus \Sigma$ 
with ${\rm dim}_{\mathcal H} \, ({\rm sing} \, V \cap \widetilde{B}) \leq n-7$ in case $n \geq 7$ or ${\rm sing} \, V \cap \widetilde{B} = \emptyset$ in case $n \leq 6$, and for every $\z \in C^{1}_{c}({\rm reg} \, V \cap \widetilde{B})$), and $V$ satisfies the no-classical-singularities hypothesis as in Theorem~\ref{reg} (or satisfies the condition 
${\mathcal H}^{n-1} \, ({\rm sing} \, V) = 0$), one can apply Theorem~\ref{reg}  or the Sheeting Theorem [\cite{Wic}, Theorem 18.2] provided only that 
$\Sigma \subset  {\rm sing} \, V$.  This is in contrast to the case when the (closed) set $\Sigma$ has locally finite $(n-2)$-dimensional Hausdorff measure, in which case $\Sigma$ is a ``removable  set''  for the stability inequality ($\star$) and hence it is not necessary that $\Sigma \subset {\rm sing} \, V$. 

In the context of the proof of Theorem~\ref{disjoint} described above, since stability of $W_{1},$  $W_{2}$ can a priori be verified only away from ${\rm spt} \, \|V_{1}\| \cap {\rm spt} \, \|V_{2}\|$ and  ${\rm spt} \, \|V_{1}\| \cap {\rm spt} \, \|W_{1}\|$ respectively, this means that it is indeed necessary that 
${\rm spt} \, \|V_{1}\| \cap {\rm spt} \, \|V_{2}\| \subset {\rm sing} \, W_{1}$ and  
${\rm spt} \, \|V_{1}\| \cap {\rm spt} \, \|W_{1}\|  \subset  {\rm sing} \, W_{2};$ here $V_{1}$, $V_{2}$ are the stationary varifolds as in Theorem~\ref{disjoint}, and $W_{1}$, $W_{2}$ are the interposed stable varifolds described above. As indicated above, these inclusions do indeed hold.

\section{strong maximum principle: proof of theorem~\ref{maxm}}\label{maxm-proof}

For the proof of Theorem~\ref{maxm} and subsequently, we shall need the following direct consequence of stationarity:

\begin{lemma}\label{UsefulLemma}
Let $L$ be an $m$-dimensional smooth connected embedded submanifold of $N$, and let $V$ be a stationary integral $m$-varifold on $N$ with ${\rm spt} \, V  \subset L.$ Then 
$V = k|L|$ for some positive constant $k$. In particular, ${\rm spt} \, V  = L.$ 
\end{lemma}

\noindent
{\bf Remark:} This is the special case of the Constancy Theorem (\cite{S1}, Theorem 41.1) when $V$ is assumed to be integral, which is the only case we need here. In this case (or more generally, when $V$ is assumed to be rectifiable), the conclusion, as shown below, is an immediate consequence of the first variation formula.

\begin{proof} Choose $M,$ $\th$ corresponding to $V$ as in the definition of integral varifold. Define a function $\theta_{1}$ on $L$ by setting 
$\theta_{1}(x) = \theta(x)$ if $x \in L \cap M$ and $\theta_{1}(x) = 0$ if $x \in L \setminus M.$ Then for any vector field $X \in C^{\infty}_{c} \, (N)$, 
\begin{equation*}
\int_{L} {\rm div}_{L} \, X \, \theta_{1} d{\mathcal H}^{n} = \int_{L \cap M} {\rm div}_{L} \, X \, \theta\, d{\mathcal H}^{n} = \int_{M} {\rm div}_{M} \, X \, \theta \, d{\mathcal H}^{n}  = 0.
\end{equation*}
Since $L$ is smooth and connected, we may take $X$ to be appropriate  tangential vector fields to $L$ and use an approximation argument to conclude from the above that $\theta_{1} = k$, a constant,  ${\mathcal H}^{n}$-a.e.\ on $L$. But this means that 
$\theta = k$ ${\mathcal H}^{n}$-a.e.\ on $M$, ${\mathcal H}^{n}(L \setminus M) = 0$ and hence ${\mathcal H}^{n} \, (L \setminus {\rm spt} \, V)= 0$. Since ${\rm spt} \, V$ is a closed subset of $N$, it follows that ${\rm spt} \, V = L$. 
\end{proof}

\begin{proof}[Proof of Theorem~\ref{maxm}]
Let $V_{1}$, $V_{2}$ be as in Theorem~\ref{maxm}  and suppose  that ${\mathcal H}^{n-1} \, ({\rm sing} \, V_{1}) = 0.$  Suppose also that ${\rm spt} \, \|V_{1}\| \cap {\rm spt} \, \|V_{2}\| \neq \emptyset.$  We wish to show that ${\rm spt} \, \|V_{1}\| = {\rm spt} \, \|V_{2}\|.$ 

By Theorem~\ref{disjoint}, we must have that ${\mathcal H}^{n-1} \, ({\rm spt} \, \|V_{1}\| \cap {\rm spt} \, \|V_{2}\|) >0$, and consequently, there is a point $x_{0} \in {\rm reg} \, V_{1} \cap {\rm spt} \, \|V_{2}\|.$  By Theorem~\ref{SolWhi}, we then have that ${\rm reg} \, V_{1}  \cap B_{\r}(x_{0}) = {\rm spt} \, \|V_{2}\| \cap B_{\r}(x_{0})$ for some $\r > 0.$  Now let $$U = \{x \in {\rm reg} \, V_{1} \, : \, \mbox{there exists $\r_{x} > 0$ such that 
${\rm reg} \, V_{1} \cap B_{\r_{x}}(x) = {\rm spt} \, \|V_{2}\| \cap B_{\r_{x}}(x)$}\}.$$ 
Then $U$ is open in ${\rm reg} \, V_{1}$, and we have just seen that $U \neq \emptyset.$ It follows from Theorem~\ref{SolWhi}  again  that $U$ is closed relative to ${\rm reg} \, V_{1}$. We claim that ${\rm reg} \, V_{1}$ is connected, from which it follows that $U = {\rm reg} \, V_{1}$ 
and hence in particular that ${\rm spt} \, \|V_{1}\| \subset {\rm spt} \, \|V_{2}\|.$ 

To see the claim, let $M$ be a connected component of ${\rm reg} \, V_{1}.$ Then $W = |M|$ is stationary in $N \setminus {\rm sing} \, V_{1}$  since $M$ has zero mean curvature. Using a standard cut-off function argument, we show that $W$ is stationary in $N$ as follows. Fix a compact subset $K$ of $N$. By stationarity of $V_{1}$ in $N$ (more precisely, by the monotonicity formula (\cite{S1}, Section 40)), there exists $\r_{0} = \r_{0}(K, N) > 0$ and $C = C(K, V_{1}) > 0$ such that  the local area bounds  
$\|W\|(B_{\r}(x)) \leq \|V_{1}\|(B_{\r}(x)) \leq C\r^{n}$ hold for each $x \in {\rm sing} \, V_{1} \cap K$
and $\r \in (0, \r_{0}]$.  Since ${\mathcal H}^{n-1}({\rm sing} \, V_{1}) = 0$ by hypothesis, we may find, given any small $\e>0$, a non-negative function $\eta_{\e} \in C^{1}(N)$ such that $\eta_{\e} \equiv 0$ in a neighborhood of ${\rm sing} \, V_{1} \cap K,$ $\eta_{\e} \equiv 1$ in $\{x \in N \, : \, {\rm dist} \, (x, {\rm sing} \, V_{1} \cap K) > \e\}$ and 
$\int_{N} |\nabla^{N} \, \eta_{\e}| \, d\|W\| \leq C\e$ where $C = C(K)$ is independent of $\e.$ (To construct such a function, choose first a finite set of points $x_{j} \in {\rm sing} \, V_{1} \cap K$ and numbers $\r_{j} >0,$ $j=1, 2, \ldots, \ell,$ such that ${\rm sing} \, V_{1} \cap K \subset \cup_{j=1}^{\ell} B_{\r_{j}}(x_{j})$ 
and $\sum_{j=1}^{\ell} \r_{j}^{n-1} < \e^{n-1}$. For each $j \in \{1, 2, \ldots, \ell\},$  let 
$\varphi_{j} \in C^{1} (N)$ be such that $\varphi_{j} \equiv 0$ in $B_{\r_{j}/2}(x_{j})$, 
$\varphi_{j} \equiv 1$ in $N \setminus B_{\r_{j}}(x_{j}),$ $0 \leq \varphi_{j} \leq 1$ and $|\nabla^{N} \, \varphi_{j}| \leq 4\r_{j}^{-1}.$ The function $\eta_{\e} = \Pi_{j=1}^{\ell} \varphi_{j}$  then has the desired properties). Now, given a vector field $X \in C^{1}_{c}(N)$, letting $\eta_{\e}$ be as above corresponding to the compact set $K = {\rm spt} \, X$, we have by stationarity of $W$ in $N \setminus {\rm sing} \, V_{1}$ that $0 = \int_{N \times G(n, k) \cap \{(x, S) \, : \, S \subset T_{x} \, N\}} {\rm div}_{S} \, \eta_{\e} X \, dW(x, S) =  \int_{N \times G(n, k) \cap \{(x, S) \, : \, S \subset T_{x} \, N\}} \eta_{\e} {\rm div}_{S} \,  X \, dW(x, S) +  \int_{N \times G(n, k) \cap \{(x, S) \, : \, S \subset T_{x} \, N\}} S(\nabla^{N} \, \eta_{\e}) \cdot X \, dW(x, S)$ where $S(\cdot)$ denotes the orthogonal projection onto $S$. Since $|\int_{N \times G(n, k) \cap \{(x, S) \, : \, S \subset T_{x} \, N\}} S(\nabla^{N} \, \eta_{\e}) \cdot X \, dW(x, S)| \leq C\e \sup \, |X|$ and $\eta_{\e}(x) \to 1$ as $\e \to 0$ for $x \in N \setminus {\rm sing} \, V_{1}$, we may let  $\e \to 0$ to conclude that $ \int_{N \times G(n, k) \cap \{(x, S) \, : \, S \subset T_{x} \, N\}} {\rm div}_{S} \,  X \, dW(x, S)=0$, i.e.\ that $W$ is stationary in $N$. 

So if ${\rm reg} \, V_{1}$ is not connected, then it has two components $M_{1}$, $M_{2}$ such that (by the preceding argument) $|M_{1}|$, $|M_{2}|$ are stationary in $N$  and $\emptyset \neq {\rm clos} \, M_{1} \cap {\rm clos} \, M_{2} \subset {\rm sing} \, V_{1}$ which by Theorem~\ref{disjoint} is impossible since ${\mathcal H}^{n-1} \, ({\rm sing} \, V_{1}) = 0.$ This proves that ${\rm reg} \, V_{1}$ is connected as claimed, and hence that ${\rm spt} \, \|V_{1}\| \subset {\rm spt} \, \|V_{2}\|.$

To complete the proof, choose, for each $x \in U  = {\rm reg} \, V_{1}$, a small number $\r_{x} > 0$ such that ${\rm reg} \, V_{1} \cap B_{\r_{x}}(x) = {\rm spt} \, \|V_{2}\| \cap B_{\r_{x}}(x).$ Note that if $x_{j} \in {\rm reg} \, V_{1}$ with $x_{j} \to x \in {\rm sing} \, V_{1}$, then $\r_{x_{j}} \to 0.$ Let $\Omega = \cup_{x \in {\rm reg} \, V_{1}} \, B_{\r_{x}/2}(x).$ Then ${\rm spt} \, \|V_{2}\| \cap \Omega= {\rm reg} \, V_{1},$ ${\rm sing} \, V_{1} \cap \Omega = \emptyset$ and ${\rm spt} \, \|V_{2}\| \cap \partial \, \Omega = {\rm sing} \, V_{1}.$   Furthermore, by Lemma~\ref{UsefulLemma} (applied with $\Omega$ in place of $N$) and connectedness of ${\rm reg} \, V_{1}$, we have that  for some positive constant $k$, $V_{2} \res \Omega = k|{\rm reg} \, V_{1}|$ as varifolds on $\Omega$.
Let $W_{2} = V_{2} \res (N \setminus {\rm clos} \, \Omega).$ Then, since $\|V_{2}\|({\rm sing} \, V_{1}) = 0$, 
it follows that $V_{2} = V_{2} \res (N \setminus {\rm sing} \, V_{1}) = W_{2} + k|{\rm reg} \, V_{1}|$ as varifolds on $N$, so that for any vector field 
$X \in C^{1}_{c}(N \setminus {\rm sing} \, V_{1})$, we have that 
\begin{eqnarray*}
0 &=& \int_{N \times G(n,k) \cap \{(x, S) \, : \, S \subset T_{x} \, N\}} {\rm div}_{S} \, X(x) \, dV_{2}(x, S)\nonumber\\ 
&=& \int_{N \times G(n, k) \cap \{(x, S) \, : \, S \subset T_{x} \, N\}} {\rm div}_{S} \, X(x) \, dW_{2}(x, S) + 
k\int_{{\rm reg} \, V_{1}} {\rm div}_{{\rm reg} \, V_{1}} \, X(x)\, d{\mathcal H}^{n}(x)\nonumber\\  
&=& \int_{N \times G(n, k) \cap \{(x, S) \, : \, S \subset T_{x} \, N\}} {\rm div}_{S} \, X(x) \, dW_{2}(x, S)
\end{eqnarray*}
where the last equality follows from the fact that ${\rm reg} \, V_{1}$ has zero mean-curvature. Thus $W_{2}$ is stationary in $N \setminus {\rm sing} \, V_{1}.$ Since ${\mathcal H}^{n-1}({\rm sing} \, V_{1}) = 0$ and $V_{2}$ is stationary in $N$, we deduce from this (by arguing as in the preceding paragraph) that $W_{2}$ is in fact stationary in $N$. (Alternatively, we may use the fact, established above, that $|{\rm reg} \, V_{1}|$ is stationary in $N$ to deduce slightly more directly that $W_{2}$ is stationary in $N$.) 

Since ${\rm spt} \, \|W_{2}\| \cap {\rm spt} \, \|V_{1}\| \subset {\rm sing} \, V_{1}$, an application of Theorem~\ref{disjoint} now tells us that ${\rm spt} \, \|W_{2}\| \cap {\rm spt} \, \|V_{1}\| = \emptyset.$ Since ${\rm spt} \, \|V_{2}\|$ is connected and $V_{2} = W_{2} + k|{\rm reg} \, V_{1}|$, we conclude from this that $W_{2} = 0$, and consequently that ${\rm spt} \, \|V_{1}\| = {\rm spt} \, \|V_{2}\|$. 
\end{proof}

\section{unique continuation: Proof of Theorem~\ref{uc}}\label{uc-proof}
We shall deduce Theorem~\ref{uc}  from Proposition~\ref{constancy} below, which is also a unique continuation result for stationary hypersurfaces and may be of independent interest. Proposition~\ref{constancy}  is an elementary consequence of well-known results from the theory of second order elliptic PDEs and stationary varifolds. We shall use  the following terminology in its statement and proof:

\noindent
{\bf Definition:} Let $N$ be a smooth manifold and $M \subset N$. We say that $M$ is \emph{strongly locally connected} if for every point $p \in {\rm clos} \, M$ and every $\r >0$, there exists 
$\s \in (0, \r)$ such that $M \cap B_{\sigma}(p)$ is connected. 

\begin{proposition}\label{constancy}
Let $V_{1}$, $V_{2}$ be stationary codimension 1 integral $n$-varifolds on a Riemannian manifold such that ${\rm reg} \, V_{1},$ ${\rm reg} \, V_{2}$ are connected and ${\rm reg} \, V_{2}$ is strongly locally connected.  If  ${\mathcal H}^{n-1+\g} \, ({\rm sing} \, V_{1}) = {\mathcal H}^{n-1+\g} \, ({\rm sing} \, V_{2}) = 0$ and ${\mathcal H}^{n-1+\g} \, ({\rm spt} \, \|V_{1}\| \cap {\rm spt} \, \|V_{2}\|) > 0$  for some $\g \in (0, 1]$, then  ${\rm spt} \, \|V_{1}\| = {\rm spt} \, \|V_{2}\|.$ In particular, ${\rm reg} \, V_{1}$ is strongly locally connected.   
\end{proposition}

\noindent
{\bf Remarks:} {\bf (1)} In Proposition~\ref{constancy}, the hypothesis that ${\rm reg} \, V_{2}$ is connected can be replaced by the hypothesis that ${\rm spt} \, \|V_{2}\|$ is connected since whenever a strongly locally connected set $M$ has its closure ${\rm clos} \, M$ connected, then $M$  itself must be connected. 

\noindent
{\bf (2)} In case $n \geq 2$, it is a well-known open question  whether stationarity of $V_{j}$ must imply ${\mathcal H}^{n-1+\g} \, ({\rm sing} \, V_{j}) = 0$ for some $\g \in (0, 1].$ Thus it is an interesting question whether Proposition~\ref{constancy}  holds without the assumption ${\mathcal H}^{n-1+\g} \, ({\rm sing} \, V_{j}) = 0$ for $j=1, 2$, even in the case $\g = 1.$

We shall give the proof of Proposition~\ref{constancy} at the end of this section. We point out the following consequence of it first. 

\begin{corollary}\label{uc-smooth}
Let $M_{1}$, $M_{2}$ be embedded smooth $n$-dimensional hypersurfaces of $N$ with locally finite mass (and possibly with 
${\rm clos} \, M_{j} \setminus M_{j} \neq \emptyset$ for $j=1$ or $2$).  If $M_{1}$, $M_{2}$ (with multiplicity 1) are stationary in $N$, $M_{1}$ is connected, $M_{2}$ is connected and strongly locally connected, and if ${\mathcal H}^{n} \, ({\rm clos} \, M_{1} \cap {\rm clos} \, M_{2}) >0$, then ${\rm clos} \, M_{1} = {\rm clos} \, M_{2}.$   
\end{corollary}

\begin{proof} For $j=1, 2$, let $V_{j} = |M_{j}|$.  By hypothesis $V_{j}$ is stationary in $N$. By smoothness of $M_{j}$, we have that $\Theta \, (\|V_{j}\|, x) = 1$ for every $x \in M_{j}$, and hence by upper semi-continuity of density, $\Theta \, (\|V_{j}\|, x) \geq 1$ for every $x \in {\rm clos} \, M_{j}.$ Since by general measure theory the upper density of $M_{j}$ at $x$ with respect to ${\mathcal H}^{n}_{\infty}$ is zero for ${\mathcal H}^{n}_{\infty}$-a.e. $x \in N \setminus M_{j}$ (\cite{S1}, Theorem 3.5), it follows that ${\mathcal H}^{n} \, ({\rm clos} \, M_{j} \setminus M_{j}) = 0$ for $j=1, 2$. Since 
${\rm spt} \, \|V_{j}\| = {\rm clos} \, M_{j}$ and ${\rm sing} \, V_{j} \subset {\rm clos} \, M_{j} \setminus M_{j}$, the corollary follows from Proposition~\ref{constancy} with $\g = 1.$ 
\end{proof}

The following lemma, which we shall need for the proof of Proposition~\ref{constancy}, is well known, and is an easy corollary of the strong unique continuation property for solutions to elliptic equations with (sufficiently) regular coefficients. 

\begin{lemma}\label{uc-pde}
Let $\Omega \subset {\mathbb R}^{n}$ be a connected, open set, and $v$ be a smooth real-valued function
solving on $\Omega$ a homogeneous uniformly elliptic linear second order partial differential equation with smooth coefficients.  If 
${\mathcal H}^{n-1+\g} \, (\{x \in \Omega \, : \, v(x) = 0\}) > 0$ for some $\g \in (0, 1]$, then $v \equiv 0$ in $\Omega.$   
\end{lemma}

\begin{proof} It follows from the implicit function theorem that the set 
$Z = \{x \in \Omega \, : \, v(x) = 0, \; Dv(x) \neq 0\}$ is an $(n-1)$-dimensional embedded submanifold of $\Omega,$ and hence in particular that ${\mathcal H}^{n-1+\g} \, (Z) = 0$. By applying this fact with 
$D^{\a} v$ in place of $v$ for each multi-index  $\a$, we deduce that the set 
$\{x \in \Omega \, : \, v(x) = 0, \; D^{\a}v(x) \neq 0 \;\;\mbox{for some multi-index $\a$}\}$ has 
${\mathcal H}^{n-1+\g}$ measure zero. Thus there is a point $x_{0} \in \Omega$ at which 
$v$ and its derivatives of all orders vanish. The lemma now follows from the well known strong unique continuation property for $v$. 
\end{proof}

\noindent
{\bf Remark.} Although the lemma as stated above suffices for our purposes here, its conclusion continues to hold under much weaker regularity hypotheses; specifically, the lemma holds under the (sharp) hypotheses that  the top order coefficients of the (divergence form) equation are locally Lipschitz, lower order coefficients are bounded and $v \in W^{1, 2}_{\rm loc} \, (\Omega)$ is a weak solution (which then, by elliptic regularity theory, automatically belongs to 
$W^{2, 2}_{\rm loc} \, (\Omega) \cap C^{1, \a} \, (\Omega)$ for any $\a \in (0, 1)$). A proof of this general version of the lemma  can be based on the monotonicity formula for the Almgren frequency function associated with $v$ (established by Garofalo and Lin in \cite{GL} and \cite{GL1}) to show: 
(a) that every blow-up of $v$ at every point $z \in Z_{v} \equiv \{x \in \Omega \, : \, v(x) = 0 \;\;{\rm and} \;\; v \not\equiv 0 \;\;\mbox{in any ball centered at} \;\; x\}$ is non-zero, and 
(b) by a dimension reducing argument and (a), that the Hausdorff dimension of $Z_{v}$ is at most $(n-1)$.  This shows, under the hypothesis ${\mathcal H}^{n-1+\g} \, (\{x \in \Omega \, : \, v(x) = 0\}) > 0$ as in the lemma, that there must exist a point near which $v$ is identically zero. Thus the set 
$\widetilde{Z}_{v} =\{x \in \Omega \, : \, \left.v\right|_{B_{\r}(x)} = 0 \;\; \mbox{for some} \;\; \r>0\}$ is non-empty and open in $\Omega$, so if $\widetilde{Z}_{v} \neq \Omega$, then we may pick a point 
$y \in \widetilde{Z}_{v}$ such that $R = {\rm dist}\, (y, \partial \, \widetilde{Z}_{v}) < {\rm dist} \, (y, \partial \, \Omega),$ choose $y_{1} \in \partial \, \widetilde{Z}_{v} \cap \Omega$ with $|y - y_{1}| = R$, and consider any blow-up $\varphi$ of $v$ at $y_{1}.$ Such $\varphi$ will have the property that $\varphi \not\equiv 0$ but $\{\varphi = 0\}$ contains a half-space, contradicting Lemma~\ref{uc-pde}. So we must have that $\widetilde{Z}_{v} = \Omega$.

\begin{proof}[Proof of Proposition~\ref{constancy}] 
Let $V_{1},$ $V_{2}$ and $\g \in (0, 1]$ be as in the statement of the proposition. 
Since by hypothesis ${\mathcal H}^{n-1+\g} \, ({\rm sing} \, V_{j}) = 0$ for $j=1, 2$ and ${\mathcal H}^{n-1+\g} \, ({\rm spt} \, \|V_{1}\| \cap {\rm spt} \, \|V_{2}\|) >0$, it follows that 
${\mathcal H}^{n-1+\g} \, ({\rm reg} \, V_{1} \cap {\rm reg} \, V_{2}) > 0$. On the other hand, the set of points $y \in  {\rm reg} \, V_{1} \cap {\rm reg} \, V_{2}$ where the tangent planes to ${\rm spt} \, \|V_{1}\|$ and ${\rm spt} \, \|V_{2}\|$ are distinct is an $(n-1)$-dimensional embedded submanifold, so in particular that set has ${\mathcal H}^{n-1+\g}$ measure zero. Hence 
the set ${\mathcal T}$ of points  $y  \in {\rm reg} \, V_{1} \cap {\rm reg} \, V_{2}$ at which 
${\rm spt} \, \|V_{1}\|$, ${\rm spt} \, \|V_{2}\|$ have a common tangent plane has positive ${\mathcal H}^{n-1+\g}$ measure, and hence (since ${\mathcal H}^{k}(A) > 0 \iff {\mathcal H}^{k}_{\infty}(A) >0$) it also has positive ${\mathcal H}^{n-1+\g}_{\infty}$ measure. 
Now let $y_{0} \in {\mathcal T}$ be a point where the upper density of ${\mathcal T}$ with respect to ${\mathcal H}^{n-1+\g}_{\infty}$  is positive. By general measure theory (e.g.\ \cite{S1}, Theorem~3.6 (2)), ${\mathcal H}^{n-1+\g}_{\infty}$-a.e.\ point in ${\mathcal T}$ is such a point.  Let $T$ be the common tangent plane to ${\rm spt} \, \|V_{1}\|$ and ${\rm spt} \, \|V_{2}\|$ at $y_{0},$ and identify $T_{y_{0}} \, N$ with ${\mathbb R}^{n+1}$ such that $T$ is identified with 
${\mathbb R}^{n} \times \{0\}.$ Let $g_{0}$ denote the exponential map at $y_{0}$, and note that for sufficiently small $\r>0$, $G_{0}^{(j)} \equiv g_{0}^{-1} \, {\rm spt} \, \|V_{j}\| \cap B_{\r}(y_{0})$ is the graph of a smooth function $u_{j}$ on $\Omega_{0}^{(j)} \equiv g_{0}^{-1} \, B_{\r}(y_{0}) \cap \pi_{0} \, G_{0}^{(j)}$ where $\pi_{0} \, : \, {\mathbb R}^{n+1} \to {\mathbb R}^{n} \times \{0\}$ is the orthogonal projection. Furthermore, on the common domain $\Omega_{0} = \Omega_{0}^{(1)} \cap \Omega_{0}^{(2)}$, each $u_{j},$ $j=1, 2$, solves  the Euler--Lagrange equation of an elliptic functional of the form ${\mathcal F}(u) = \int_{\Omega_{0}} F(Du)$, where the integrand $F$ is smooth.  It is 
standard then that $v \equiv u_{1} - u_{2}$ solves on $\Omega_{0}$ a homogeneous uniformly elliptic equation with smooth coefficients. Since $y_{0}$ is a point of positive upper density for ${\mathcal T}$ with respect to ${\mathcal H}^{n-1+\g}_{\infty}$, it follows from the definition of upper density that 
provided $\r>0$ is sufficiently small, ${\mathcal H}^{n-1+\g} \, (\{x \in \Omega_{0} \, : v(x) = 0\}) > 0$. Hence by Lemma~\ref{uc-pde}  $v \equiv 0$ on $\Omega_{0}$, which means that 
${\rm spt} \, \|V_{1}\| \cap B_{\r}(y_{0}) = {\rm spt} \, \|V_{2}\| \cap B_{\r}(y_{0})$ for suitably small $\r > 0.$ 

Let $U = \{x \in {\rm reg} \, V_{1} \, : \, \mbox{there exists $\r > 0$ such that 
${\rm spt} \, \|V_{1}\| \cap B_{\r}(x) = {\rm spt} \, \|V_{2}\| \cap B_{\r}(x)$}\}.$
By definition, $U$ is open relative to ${\rm reg} \, V_{1}$, and $U \neq \emptyset$ since $y_{0} \in U.$ $U$ is also closed relative to ${\rm reg} \, V_{1}$. To see this, let $y \in {\rm reg} \, V_{1}$ be such that 
there is a sequence of points $x_{1}, x_{2}, \ldots \in U$ with $x_{j} \to y.$ Then $y \in {\rm spt} \, \|V_{2}\|.$ Choose small $\r > 0$ such that 
${\rm spt} \, \|V_{1}\| \cap B_{\r}(y) \subset {\rm reg} \, V_{1}.$ Since ${\rm reg} \, V_{2}$ is strongly locally connected, 
there exists $\s \in (0, \r)$ such that ${\rm reg} \, V_{2} \cap B_{\s}(y)$ is connected. Let $V_{j}^{\s}  = V_{j} \res B_{\s}(y)$ for $j=1, 2$, and 
let $$U_{\s} = \{x \in {\rm reg} \, V_{2}^{\s} \, : \, \mbox{there exists $\r > 0$ such that 
${\rm spt} \, \|V_{2}^{\s}\| \cap B_{\r}(x) = {\rm spt} \, \|V_{1}^{\s}\| \cap B_{\r}(x)$}\}.$$ 
Then $U_{\s}$ is open relative to 
${\rm reg}\, V_{2}^{\s}$ and is non-empty since $x_{j} \in U_{\s}$ for all sufficiently large $j.$  Since ${\rm sing} \, V_{1}^{\s} = \emptyset$, it follows that at any limit point $z$ of $U_{\s}$ in ${\rm reg} \, V_{2}^{\s},$ both ${\rm spt} \, \|V_{1}^{\s}\|$ and  
${\rm spt} \, \|V_{2}^{\s}\|$ have the same tangent plane, so writing ${\rm spt} \, \|V_{1}^{\s}\|,$  ${\rm spt} \, \|V_{2}^{\s}\|$ near $z$ as graphs of functions $u_{1}^{\s}$, $u_{2}^{\s}$ defined on a domain in this common tangent plane, and noting that $v_{\s} \equiv u_{1}^{\s} - u_{2}^{\s}$ vanishes on a non-empty open set, we conclude with the help of Lemma~\ref{uc-pde} as in the paragraph above that $z \in U_{\s}.$ By connectedness of ${\rm reg} \, V_{2}^{\s}$, we then have that 
${\rm reg} \, V_{2}^{\s} \subset {\rm spt}\, \|V_{1}^{\s}\|$ and hence that ${\rm spt} \, \|V_{2}^{\s}\| \subset {\rm spt} \, \|V_{1}^{\s}\|,$ which implies, by Lemma~\ref{UsefulLemma}, that 
${\rm spt} \, \|V_{2}^{\s}\| = {\rm spt} \, \|V_{1}^{\s}\|.$ Thus $y \in U$ so $U$ is closed relative to ${\rm reg} \, V_{1}$ as claimed.

Since ${\rm reg} \, V_{1}$ is connected, we conclude that ${\rm reg} \, V_{1} \subset {\rm reg} \, V_{2}.$  
We claim that this implies that ${\rm reg} \, V_{1}$ is strongly locally connected. To see this, note first that ${\rm spt} \, \|V_{1}\| \subset {\rm spt} \, \|V_{2}\|,$ so
if $z \in {\rm spt} \, \|V_{1}\| \cap {\rm reg} \, V_{2},$ then we have by Lemma~\ref{UsefulLemma} that ${\rm spt} \, \|V_{1}\| \cap B_{\s}(z) = {\rm spt} \, \|V_{2}\| \cap B_{\s}(z)$ for sufficiently small $\s>0$ and consequently that $z \in {\rm reg} \, V_{1}$. Thus  ${\rm spt} \, \|V_{1}\| \cap {\rm reg} \, V_{2} \subset {\rm reg} \, V_{1}$, or, equivalently, ${\rm sing} \, V_{1} \subset {\rm sing} \, V_{2}$. Now 
let $z \in {\rm sing} \, V_{1}$ and $\r>0$. Then $z \in {\rm sing} \, V_{2}$ so by strong local connectedness of ${\rm reg} \, V_{2}$, there exists $\s  \in (0, \r)$ such that ${\rm reg} \, V_{2} \cap B_{\s}(z)$ is connected.  But since ${\rm sing} \, V_{1} \subset {\rm sing} \, V_{2}$ and ${\rm reg} \, V_{1} \subset {\rm reg} \, V_{2}$, it follows that ${\rm reg} \, V_{1} \cap B_{\s}(z)$ is both an open and a closed subset of ${\rm reg} \, V_{2} \cap B_{\s}(z).$ Thus ${\rm reg} \, V_{1} \cap B_{\s}(z) = {\rm reg} \, V_{2} \cap B_{\s}(z)$, and in particular ${\rm reg} \, V_{1} \cap B_{\s}(z)$ is connected. This means that ${\rm reg} \, V_{1}$ is strongly locally connected as claimed.

Since ${\rm reg}\, V_{2}$ is connected by hypothesis, we can now repeat the argument leading to the conclusion  ${\rm reg} \, V_{1} \subset {\rm reg} \, V_{2},$ with the roles of $V_{1}$, $V_{2}$ reversed, to deduce that ${\rm reg} \, V_{1} = {\rm reg}\, V_{2},$ and consequently that ${\rm spt} \, \|V_{1}\| = {\rm spt} \, \|V_{2}\|.$ 
\end{proof}

\begin{proof}[Proof of Theorem~\ref{uc}]
Since by hypothesis ${\rm spt} \, \|V_{j}\|$ is connected and ${\mathcal H}^{n-1} \, ({\rm sing} \, V_{j}) = 0$ for $j=1, 2$, it follows from Theorem~\ref{disjoint} that ${\rm reg} \, V_{j}$ is connected for $j=1, 2$. (The argument here is exactly that in the third paragraph of the proof of Theorem~\ref{maxm}.) Also, for each $p \in {\rm spt} \, \|V_{2}\|,$ we may choose $\s > 0$ sufficiently small such that ${\rm spt} \, \|V_{2}\| \cap B_{\s}(p)$ is connected, and use Theorem~\ref{disjoint} again with $V_{2} \res B_{\s}(p)$ in place of $V$ and $B_{\s}(p)$ in place of $N$ to see that ${\rm reg} \, V_{2} \cap B_{\s}(p)$ is connected. Thus ${\rm reg} \, V_{2}$ is strongly locally connected. The theorem now follows from Proposition~\ref{constancy}. 
\end{proof}

\end{document}